\documentclass[12pt]{amsart}

\usepackage{amsfonts}
\usepackage{amsmath}
\usepackage{amssymb,latexsym}
\usepackage{amsthm}
\usepackage{mathrsfs}
\usepackage{bm}
\usepackage{enumerate}
 \usepackage{color}

\numberwithin{equation}{section}

\renewcommand{\i}{\mathrm{i}}
\renewcommand{\vec}[1]{\ensuremath{\mathchoice
                    {\mbox{\boldmath$\displaystyle#1$}}
                    {\mbox{\boldmath$\textstyle#1$}}
                    {\mbox{\boldmath$\scriptstyle#1$}}
                    {\mbox{\boldmath$\scriptscriptstyle#1$}}}}%

\theoremstyle{plain}
\newtheorem{theorem}[equation]{\bf Theorem}

\newtheorem{lemma}[equation]{\bf Lemma}

\theoremstyle{definition}
\newtheorem{definition}[equation]{\bf Definition}
\theoremstyle{remark}
\newtheorem{remark}[equation]{\bf Remark}
\newcommand{\llg}{\langle\hskip -1.5pt\langle}
\newcommand{\rrg}{\rangle\hskip -1.5pt\rangle}

\begin{document}

\title{Some New Addition Formulae for Weierstrass Elliptic Functions}

\author{J. Chris Eilbeck$^{1}$, Matthew England$^{2}$ and Yoshihiro
  \^Onishi$^{3}$}

\footnote{
Department of Mathematics and the Maxwell Institute for Mathematical Sciences, Heriot-Watt University, 
Edinburgh, EH14 4AS, United Kingdom\\
\hspace*{0.15in} $^{2}$Department of Computer Science, University of Bath, 
Bath, BA2 7AY, United Kingdom\\
\hspace*{0.15in} $^{3}$Department of Mathematics, Faculty of Sciences and Technology, Meijo University, 
1-501~\,Shiogamaguchi, Tenpaku-ku, Nagoya, 468-8502, Japan}

\begin{abstract}
We present new addition formulae for the Weierstrass functions associated with a general elliptic curve.  
We prove the structure of the formulae in $n$-variables and give the explicit addition formulae for the $2$- and $3$-variable cases.  
These new results were inspired by new addition formulae found in the case of an equianharmonic curve, 
which we can now observe as a specialisation of the results here.  
The new formulae, and the techniques used to find them,  
also follow the recent work for the generalisation of Weierstrass' functions to curves of higher genus.
\end{abstract}

\maketitle

\section{Introduction} 
\label{SEC:Intro}

This paper concerns new addition formulae for the Weierstrass functions associated with 
the general elliptic curve \(f(x,y)=0\) with
\[
f(x,y)=y^2+(\mu_1x+\mu_3)y-(x^3+\mu_2x^2+\mu_4x+\mu_6). 
\]
We describe a new class of formulae in Theorem \ref{main_gen} and derive explicit examples in Theorems \ref{main_thm} and \ref{main_thm3}. 

Our work follows both classical results for the Weierstrass elliptic curve, and recent work for the equianharmonic case (as well as specialised higher genus curves).  
We summarise these results respectively in Subsections \ref{SUBSEC:Weierstrass} and \ref{SUBSEC:Specialised} 
before giving our inspiration and motivation in Subsection \ref{SUBSEC:Motivation}.

\subsection{The Weierstrass elliptic curve}
\label{SUBSEC:Weierstrass}

\noindent Consider the Weierstrass equation
\begin{equation}
\label{weier_ellipt}
{\wp'(u)}^2 = 4\wp(u)^3 - g_2\wp(u) - g_3, 
\end{equation}
where $g_2$ and $g_3$ are the elliptic invariants and $\wp(u)=-\frac{d^2}{du^2}\log\sigma(u)$, 
$\sigma(u)$ the famous functions of Weierstrass (see for example Chapter 20 of \cite{ww}).  
There is an especially well-known addition formula (see for instance p.451 of \cite{ww})
\begin{equation}
\label{2term_fs}
- \frac{\sigma(u+v)\sigma(u-v)}{\sigma(u)^2\sigma(v)^2}
=\wp(u)-\wp(v).
\end{equation}
Also, for $n$ variables 
$u^{(j)}, j=1 \dots n$, 
it is known that
\begin{equation}
\label{det_formula}
\begin{aligned}
&(-1)^{(n-1)(n-2)/2}\frac{\sigma(\sum_{i=1}^nu^{(i)})
\prod_{i<j}\sigma(u^{(i)}-u^{(j)})}
{\prod_{j}\sigma(u^{(j)})^n} \\
&\quad =\frac{1}{\prod\limits_{j=1}^{n-1}j!}
\left|\,
\begin{matrix}
1      & \wp(u^{(1)}) & \wp'(u^{(1)}) & \cdots & \wp^{(n-2)}(u^{(1)}) \\
1      & \wp(u^{(2)}) & \wp'(u^{(2)}) & \cdots & \wp^{(n-2)}(u^{(2)}) \\
\vdots & \vdots       & \vdots        & \ddots & \vdots             \\
1      & \wp(u^{(n)}) & \wp'(u^{(n)}) & \cdots & \wp^{(n-2)}(u^{(n)})
\end{matrix}\,
\right|.
\end{aligned}
\end{equation}
This and other addition formulae may be found on p.458 of \cite{ww}, for example.  
These formulae are a reflection of the involution of the elliptic curve defined by (\ref{weier_ellipt}):
\begin{equation}\label{alg_eq}
 \mathscr{C}\,:\, y^2=x^3-\tfrac{g_2}{4}x-\tfrac{g_3}{4}, \qquad (y=\tfrac12\wp'(u), \  x=\wp(u)).
\end{equation}

\subsection{Specialised curves}
\label{SUBSEC:Specialised}

The equianharmonic case is when the elliptic invariant $g_2=0$ (and
$g_3$ is assumed non-zero).  In this case there is a three-term
analogy of equation (\ref{2term_fs}) which reflects the cyclic
automorphism group of order three.  Let $\zeta$ be a primitive cube
root of unity (without loss of generality we may take
$\zeta=(-1+\sqrt{-3})/2$).  Then the Weierstrass functions specialised
to this case satisfy
\begin{equation}\label{emo_fs}
 - \frac{\sigma(u+v)\sigma(u+{\zeta}v)\sigma(u+{\zeta}^2v)}
  {\sigma(u)^3\sigma(v)^3}
  =\frac12 \left(\wp'(u)+\wp'(v)\right).  
\end{equation}
This was derived recently as Proposition 5.1 of \cite{emo}. 

The authors have derived similar formulae for specialised higher genus
curves and functions. Generalisations of Weierstrass' functions may be
defined following the work of Klein (see for example \cite{BEL97,
  EEL00}) which satisfy formulae generalising equations
(\ref{weier_ellipt}) and (\ref{2term_fs}).  In the case of trigonal
curves, the authors found that further addition formula after making
specialisations of the curve parameters in analogy with the
equianharmonic case.  Results for genus three were given in Theorem
10.1 of \cite{eemop} and Theorem 5.4 in \cite{eeo11}), and for genus
four in Theorem 8 in \cite{England11}.

\subsection{Aim and motivation}
\label{SUBSEC:Motivation}

The aim of this paper is to introduce generalisations of (\ref{det_formula}) and (\ref{2term_fs}), 
which are both beyond (\ref{emo_fs}) and for 
the most general elliptic curve (\ref{ellipt_curve}) rather than a specialised curve.

Our approach is inspired by the following observation.
Let us take a map 
\begin{equation}\label{coverings}
  \varphi : \mathscr{C}\longrightarrow \mathbb{P}^1,
\end{equation}
where \(\mathbb{P}^1\) denotes the projective line.  For technical
reasons, we assume \(\varphi\) is a polynomial of \(x\) and \(y\).  We
regard \(\mathscr{C}\) as a complex torus \(\mathbb{C}/\Lambda\) and
\(\varphi\) as a function on \(\mathbb{C}\) with the set of periods
\(\Lambda\).  Let \(u\) be a variable on \(\mathbb{C}\) and let us
take the set of variable points \(S=\varphi^{-1}(\varphi(u))\).  If
\(\varphi=x\), then \(S=\{u, -u\}\) and this choice gives rise to
(\ref{2term_fs}) and (\ref{det_formula}).  If \(\varphi=y\) and
\(\mu_1=\mu_2=\mu_3=\mu_4=0\), then \(S=\{u,\zeta{u},\zeta^2{u}\}\)
and this choice gives rise to (\ref{emo_fs}), where \(\zeta\) is a
primitive cube root of unity.  It is natural to investigate the
generalisations of these formulae for arbitrary \(\varphi\), and the
present paper is a first attempt at this problem.  


We define the objects we work with formally in Sections
\ref{SEC:Preliminaries} and \ref{SEC:conjpoints}.  Then in Section
\ref{SEC:AFN} we describe the structure of a new class of addition
formulae in $n$ arbitrary variables, with explicit expressions derived
for the cases \(n=2\) and \(n=3\) in Sections \ref{SEC:AF2} and
\ref{SEC:AF3} respectively.  Finally in Section \ref{SEC:Final} we
describe some possibilities for extending these ideas.


Our personal interest in such formulae stems from their beauty,
but they may also have important applications.  The classical theory
of the elliptic functions has of course been widely applied (see for
example \cite{AE06} and \cite{Lawden80} for details on applications to
geometry, algebra, arithmetic, mechanics, statistics).  The
Weierstrass functions in particular give solutions to many systems,
including the spherical pendulum, the spinning top and the KdV
equation for water waves.

The addition formula of the functions are algebraic analogues of the
well known addition law for points on the elliptic curve, fundamental
to elliptic curve cryptography.  The addition formula can be
particularly important in number theory (see for example
\cite{Lang78}).  More recently, in \cite{GV12}, null geodesics in
Schwarzschild spacetime were described by the Weierstrass
$\wp$-function and the addition formula used to connect the values of
radial distance at different points on the geodesic.

Also, the recent work on the generalisation of these functions to
higher genus curves has begun to find applications, including:
describing the double pendulum \cite{EPR03}; solutions to systems in
the KP hierarchy (see for example \cite{BEL97, EE09}); reductions of
the Benney equations (see for example \cite{BG03, EG09}); and
describing geodesics in black hole space times (see for example
\cite{EHKKL11, EHKKL12}).

\section{Preliminaries} 
\label{SEC:Preliminaries} 

\noindent  
The reader is referred to \cite{o} for more details of the material in
this section.  Define
\begin{equation}\label{ellipt_curve}
  f(x,y)=y^2+(\mu_1x+\mu_3)y-(x^3+\mu_2x^2+\mu_4x+\mu_6). 
\end{equation}
We consider the general elliptic curve $\mathscr{C}$ defined by
$f(x,y)=0$ with the unique point $\infty$ at infinity.  Although we
assume $\mathscr{C}$ is non-singular, the formulae in our theorems are
valid even if this is not the case.  It is known that any elliptic
curve over any perfect field is written in this form (see Chapter 8 of
\cite{cassels}, Chapter 3.3 of \cite{silverman}).
Many of the results for this curve are valid as identities on
power series over quite general base rings.  In this paper we
henceforth work over $\mathbb{C}$.

We may define weights, denoted \(\mathrm{wt}\), by
\begin{equation*}
  \mathrm{wt}(x)=-2, \ \ 
  \mathrm{wt}(y)=-3, \ \ 
  \mathrm{wt}(\mu_j)=-j.
\end{equation*}
From this definition, it is possible to deduce a weight for every
object in the paper such that every formula in the paper is of
homogeneous weight.  In general a numerical subscript throughout this
paper will refer to the corresponding (negative) weight, except for
the classical constants $g_2$ and $g_3$, which have weight $-4$ and
$-6$ respectively.

Any differential of the first kind is a constant multiple of
\begin{equation*}
  \omega=\omega(x,y)=\frac{dx}{f_y(x,y)}=\frac{dx}{2y+(\mu_1x+\mu_3)}
  =-\frac{dy}{f_x(x,y)},
\end{equation*}
where \(f_y\) and \(f_x\) denote \(\frac{\partial}{\partial y}f\) and
\(\frac{\partial}{\partial x}f\) respectively.  Let \(\Lambda\) denote
the lattice consisting of the integrals of this differential along any
closed path:
\begin{equation*}
\Lambda=\bigg\{\oint\omega\bigg\}.
\end{equation*}
We define two meromorphic functions \(x(u)\) and \(y(u)\) by the set of equalities
\begin{equation}\label{xy}
u=\int_{\infty}^{(x(u),y(u))}\!\!\!\omega, \quad f\big(x(u),y(u)\big)=0.
\end{equation}
Clearly, these are periodic with respect to \(\Lambda\) and have poles only at the points in \(\Lambda\).  
Note that it follows from these definitions that the variable \(u\) is of weight 1: \(\mathrm{wt}(u)=1\).

From the definitions in (\ref{xy}) we have
\begin{equation*}
x(-u)=x(u), \quad y(-u)=y(u)+\mu_1x(u)+\mu_3.
\end{equation*}
Both \(x(u)\) and \(y(u)\) have a pole only at \(u=0\), of order 2 and 3 respectively. 

Let us take a local parameter \(t\) around the point \(\infty\) satisfying
\begin{equation} \label{tdef}
y=\frac1{t^3}.
\end{equation}
This choice of a local parameter is different from the usual one: $t=-x/y$. 
Using (\ref{tdef}) and (\ref{xy}), we can obtain the power series expansions of \(x(u)\) and \(y(u)\) beginning with
\begin{equation}\label{2.024}
  \begin{aligned}
    x(u)&=u^{-2}-(\tfrac{1}{12}{\mu_1}^2 +\tfrac13{\mu_2}) \\
    &\qquad +(\tfrac1{240}{\mu_1}^4 +\tfrac{1}{30}{\mu_2}{\mu_1}^2
    -\tfrac1{10}{\mu_3}{\mu_1}
    +\tfrac1{15}{\mu_2}^2 -\tfrac15{\mu_4})u^2+\cdots, \\
    y(u)&=-u^{-3}-\tfrac12{\mu_1}u^{-2}+(\tfrac1{24}{\mu_1}^3
    +\tfrac16{\mu_2}{\mu_1}-\tfrac12{\mu_3})+\cdots.
  \end{aligned}
\end{equation}
For two variable points \((x,y)\) and \((z,w)\) on  \(\mathscr{C}\),  we define
\begin{equation*}
  \Omega(x,y,z,w)
  =\frac{(y+w+\mu_1z+\mu_3)dx}{(x-z)(2y+\mu_1x+\mu_3)}.
\end{equation*}
This has a pole of order 1 with residue 1 at \((z,w)\) when regarded as a form with variable \((x,y)\) for a fixed \((z,w)\).  
Indeed, since \((2w+\mu_1z+\mu_3)=f_y(z,w)\) when \((x,y)=(z,w)\), the residue at \((z,w)\) is 1, and the zeroes of the numerator and denominator at \((x,y)=(z,-w-\mu_1z-\mu_3)\) cancel.

For a differential \(\eta\) of the 2nd kind with pole only at
\(\infty\), we take
\begin{equation*}
  \vec{\xi}(x,y;z,w) =\tfrac{d}{dz}\Omega(x,y;z,w)dz 
  -\omega(x,y)\eta(z,w),
\end{equation*}
where \((x,y)\), \((z,w)\in\mathscr{C}\). 
Then the differential of the second kind \(\eta\) is 
chosen so that it satisfies
\begin{equation*}
\vec{\xi}(x,y;z,w)=\vec{\xi}(z,w;x,y). 
\end{equation*}
Such choice of a differential form \(\eta\) is not unique. 
In this paper, we chose 
\begin{equation}\label{1.49}
\eta(x,y)=\frac{-xdx}{2y+\mu_1x+\mu_3}
\end{equation}
(see \cite{o}).
We fix the notation \(\eta\) for the form (\ref{1.49}) from now on. 
Let \(\alpha\) and \(\beta\) be a pair of closed paths on \(\mathscr{C}\) 
which represent a symplectic basis of the homology group \(H_1(\mathscr{C},\mathbb{Z})\).  
We let \(\omega'\) and \(\omega''\) be periods of \(\omega\) with respect to the closed paths \(\alpha\) and \(\beta\).  
Similarly, let \(\eta'\) and \(\eta''\) be periods of \(\eta\) with respect to \(\alpha\) and \(\beta\).  
In general, for a given \(v\in\mathbb{C}\), we denote by \(v'\) and \(v''\) the real numbers such that
\begin{equation*}
v=v'\omega'+v''\omega''.   
\end{equation*}
\begin{definition}
\label{def:gensig}
We define the \emph{sigma function} by
\begin{equation*}
  \sigma(u)=u\,\exp\bigg\{{-}\int_0^u\int_0^u\bigg(x(u)-\frac{1}{u^2}\bigg)du\bigg\}. 
\end{equation*}
The integrals and the exponential should be regarded as operations for power series. 
We can also express \(\sigma(u)\) analytically with $\theta$-functions as 
\begin{equation*}
  \sigma(u)={\eta_{\mathrm{\,Ded}}({\omega'}^{-1}\omega'')}^{-3}
  {\cdot}\frac{\,\omega'}{\,2\pi}{\cdot}
  \exp\big(-\tfrac12 u^2\eta'{\omega'}^{-1}\big)
  \vartheta\bigg[\begin{matrix}\tfrac12 \\\tfrac12\end{matrix}
  \bigg]({\omega'}^{-1}u\big|{\omega'}^{-1}\omega''),
\end{equation*}
where \(\eta_{\mathrm{\,Ded}}(\tau)=e^{\frac{\pi\i\tau}{12}}\prod_{n=1}^{\infty}(1-e^{2\pi\i\tau{n}})\) 
is Dedekind's eta function (see \cite{o}). 
\end{definition}
It is known that the \(\sigma\)-function does not depend on the choice
of symplectic basis \(\alpha\) and \(\beta\) of
\(H_1(\mathscr{C},\mathbb{Z})\), and it can be easily checked that
\begin{equation}\label{oddness}
\sigma(-u)=-\sigma(u).
\end{equation}
Let
\begin{equation*}
L(u,v)=u(v'\eta'+v''\eta'')
\end{equation*}
for \(u\) and \(v\in\mathbb{C}\), and 
$\chi(\ell) = {\rm exp}(2\pi{\rm i}(\tfrac{1}{2}\ell' - \tfrac{1}{2}\ell'' + \tfrac{1}{2}\ell'\ell'')).$
Then the $\sigma$-function has the following quasi-periodicity property.
\begin{lemma}[Lemma 2.6 in \cite{o}]
\label{translation} 
The $\sigma$-function satisfies
\begin{equation}\label{translational}
  \sigma(u+\ell)=\chi(\ell)\sigma(u)\exp L(u+\tfrac12\ell,\ell) 
  \ \ \ (\ell\in\Lambda).
\end{equation}
\end{lemma}
\noindent Let ${\overline{\mu}_1} = \mu_1/2$.  Then the \(\sigma\)-function may be represented by a series expansion starting with
\begin{equation}\label{2.10}
\begin{aligned}
  \sigma(u)&=u +({\overline{\mu}_1}^2 + \mu_2)(\tfrac{1}{3!})u^3 \\
  &\qquad +({\overline{\mu}_1}^4 + 2{\mu_2}{\overline{\mu}_1}^2
  + \mu_3\mu_1 + {\mu_2}^2 + 2\mu_4)(\tfrac{1}{5!})u^5 \\
  &\qquad + ({\overline{\mu}_1}^6 + 3{\mu_2}{\overline{\mu}_1}^4
  + 6\mu_3{\overline{\mu}_1}^3 + 3{\mu_2}^2{\overline{\mu}_1}^2
  + 6\mu_4{\overline{\mu}_1}^2
  \\  &\qquad
  + 6\mu_3\mu_2{\overline{\mu}_1} + {\mu_2}^3 + 6\mu_4\mu_2
  + 6{\mu_3}^2 + 24\mu_6)(\tfrac{1}{7!})u^7  +\cdots,
\end{aligned}
\end{equation}
For simplicity, we use 
\(\mathbb{Z}[\mu_1,\mu_2,\mu_3,\mu_4,\mu_6]=\mathbb{Z}[\vec{\mu}]\),
\(\mathbb{Q}[\mu_1,\mu_2,\mu_3,\mu_4,\mu_6]=\mathbb{Q}[\vec{\mu}]\), 
and
\(\mathbb{Z}[\overline{\mu}_1,\mu_2,\mu_3,\mu_4,\mu_6] = 
\mathbb{Z}[\vec{\mu}']\).
\par
For a commutative ring \(R\), we denote by \(R\llg{z}\rrg\) the ring
\begin{equation*}
  \bigg\{\,\sum_{j=0}^{\infty}a_j\dfrac{z^j}{j!}\,\Big|\,a_j\in R\,\bigg\}. 
\end{equation*}
Each element of this ring is said to be {\it Hurwitz integral} over \(R\). 
\begin{remark}[Hurwitz integrality]
The expansion (\ref{2.10}) is Hurwitz integral over \(\mathbb{Z}[\vec{\mu}']\):
\begin{equation}\label{hurwitz1}
  \sigma(u)\in\mathbb{Z}[\vec{\mu}']\llg{u}\rrg.
\end{equation}
However, it is also known that 
\begin{equation}\label{hurwitz2}
  \sigma(u)^2\in\mathbb{Z}[\vec{\mu}]\llg{u}\rrg. 
\end{equation}
The reader is referred to the discussion in \cite{o}.  This
integrality of the coefficients of this expansion is 
implicitly taken up in Remark \ref{int_coeff} later.  In this paper, we need only
the fact that \(A_n\in\mathbb{Q}[\vec{\mu}]\).
\end{remark}
\begin{definition}
\label{def:wp}
We now define as usual the elliptic functions
\begin{equation}\label{wp}
\wp(u)=-\frac{d^2}{du^2}\log\sigma(u), \ \ \ 
\wp'(u)=\frac{d}{du}\wp(u). 
\end{equation}
\end{definition}
These are periodic for any period $\ell \in \Lambda$ by Lemma
\ref{translation}.  Also, by (\ref{oddness}), we have
\begin{equation}\label{parity}
\wp(-u)=\wp(u) \ \ \ \mbox{and}\ \ \ \wp'(-u)=-\wp'(u).
\end{equation}

This \(\wp(u)\) for the general curve (\ref{ellipt_curve}) is slightly
different from the work of Weierstrass for (\ref{alg_eq}).  Our
$\wp(u)$ has the expansion
\begin{equation*}
  \wp(u)=\frac1{u^2}
  +\sum_{\ell\in\Lambda,\ell\neq0}
  \Big(\frac1{(u-\ell)^2}-\frac1{\ell^2}\Big)-\frac{{\mu_1}^2+4\mu_2}{12},
\end{equation*}
which is shown by the positions of the zeroes of \(\sigma(u)\).
Comparing the power series expansions in (\ref{2.024}) and the
essential part of the expansion of \(\wp(u)\) with respect to \(u\)
obtained by (\ref{2.10}), we have
\begin{equation}\label{P_and_xy}
\wp(u)=x(u),  \ \ \ \mbox{and}\ \ \     
\wp'(u)=2y(u)+\mu_1x(u)+\mu_3.
\end{equation}
Note that the $\sigma$-function has weight $+1$ and the $\wp$-function
weight $-2$.   
By (\ref{P_and_xy}), we see that
\begin{equation}\label{xy_by_sigma}
 \begin{aligned}
  x(u)&=\frac{\sigma''(u)\sigma(u)-\sigma'(u)^2}
             {\sigma(u)^2} \ \ \mbox{and} \\
  y(u)&=\frac{-\tfrac12\sigma'''(u)\sigma(u)^2+\tfrac32\sigma''(u)
    \sigma'(u)\sigma(u)-\sigma'(u)^3}{\sigma(u)^3} \\
  &\qquad -\frac{\mu_1}2{\cdot}\frac{\sigma''(u)\sigma(u)-\sigma'(u)^2}
  {\sigma(u)^2}+\frac{\mu_3}2,
 \end{aligned} 
\end{equation}
where \(\sigma'(u)=\frac{d}{du}\sigma(u)\),
\(\sigma''(u)=\frac{d^2}{du^2}\sigma(u)\), and
\(\sigma'''(u)=\frac{d^3}{du^3}\sigma(u)\).
If the parameters in (\ref{ellipt_curve}) take values
\(\mu_1=\mu_2=\mu_3=0\), \(\mu_4=-\tfrac{1}{4}g_2\),
\(\mu_6=-\tfrac{1}{4}g_3\), then the function \(\wp(u)\) from
Definition \ref{def:wp} satisfies the classical equation
(\ref{weier_ellipt}), and the function \(\sigma(u)\) from Definition
\ref{def:gensig} is exactly the same as the Weierstrass
$\sigma$-function. Under this specialisation the results of this
section map to the well-known results for the Weierstrass functions.

\section{Conjugate points and variables}
\label{SEC:conjpoints} 

For a variable point \((x,y)\) on \(\mathscr{C}\), we have three
points (up to multiplicity) with the same second coordinate \(y\).  We
denote these \emph{conjugate points} by
\begin{equation*}
(x,y),\ \ \ (x^{\star},y),\, \mbox{ and } \, (x^{\star\star},y). 
\end{equation*}
Moreover, for
\begin{equation}\label{nostar}
v=\int_{\infty}^{(x,y)}\omega,
\end{equation}
we define
\begin{equation}\label{stars}
v^{\star}=\int_{\infty}^{(x^{\star},y)}\omega, \, \mbox{ and } \,  
v^{\star\star}=\int_{\infty}^{(x^{\star\star},y)}\omega. 
\end{equation}
Here the paths of integration are defined as the continuous transformations 
by taking \({\,}^\star\) or \({\,}^{\star\star}\) for all points on the path in (\ref{nostar}).  
We call $v, v^{\star}, v^{\star\star}$ \emph{conjugate variables}.

\begin{lemma}\label{sum_of_v}
In the above notation we have 
\begin{equation}\label{star_relation}
v+v^{\star}+v^{\star\star}=0.
\end{equation}
\end{lemma}

\begin{proof} 
Since, for a given \(y\), the \(x\), \(x^{\star}\), \(x^{\star\star}\) are the solution of the equation \(f(X,y)=0\) of \(X\), we see \(f(X,y)=-(X-x)(X-x^{\star})(X-x^{\star\star})\).  So
\begin{equation*}
\begin{aligned}
f_x(x,y)&=-(x-x^{\star})(x-x^{\star\star}), \\
f_x(x^{\star},y)&=-(x^{\star}-x)(x^{\star}-x^{\star\star}), \\
f_x(x^{\star\star},y)&=-(x^{\star\star}-x)(x^{\star\star}-x^{\star}).
\end{aligned}
\end{equation*}
Then since
\begin{equation*}
\frac{1}{(x-x^{\star})(x-x^{\star\star})}
+\frac{1}{(x^{\star}-x)(x^{\star}-x^{\star\star})}
+\frac{1}{(x^{\star\star}-x)(x^{\star\star}-x^{\star})}=0,
\end{equation*}
we find
\begin{equation*}
-\frac{dy}{f_x(x,y)}
-\frac{dy}{f_x(x^{\star},y)}
-\frac{dy}{f_x(x^{\star\star},y)}=0.
\end{equation*}
This implies that
\begin{equation*}
\int_{\infty}^{(x,y)}\Big(
\frac{dy}{f_x(x,y)}
+\frac{dy}{f_x(x^{\star},y)}
+\frac{dy}{f_x(x^{\star\star},y)}\Big)=0. 
\end{equation*} Hence
\begin{equation*}
 \int_{\infty}^{(x,y)}\frac{dy}{f_x(x,y)}
+\int_{\infty}^{(x^{\star},y)}\frac{dy}{f_x(x,y)}
+\int_{\infty}^{(x^{\star\star},y)}\frac{dy}{f_x(x,y)}=0,
\end{equation*}
where the three paths of integrals are chosen as in (\ref{nostar}) and (\ref{stars}). 
Now we have the desired equality. 
\end{proof}
Note that if \(\ell\in\Lambda\) then \(\ell^{\star}\), \(\ell^{\star\star}\in\Lambda\), 
and that \(\ell+\ell^{\star}+\ell^{\star\star}=0\) by Lemma \ref{sum_of_v}.

\begin{remark}
In the Weierstrass case when the parameters in (\ref{ellipt_curve})
take values \(\mu_1=\mu_2=\mu_3=0\), \(\mu_4=-\tfrac{1}{4}g_2\),
\(\mu_6=-\tfrac{1}{4}g_3\), then we have
\(\wp'(v)=\wp'(v^{\star})=\wp'(v^{\star\star})\).
\end{remark}

\noindent Using the curve equation (\ref{ellipt_curve}) and a local parameter
(\ref{tdef}) we may obtain an expansion
\begin{equation}\label{x_exp}
\begin{aligned}
  x&=t^{-2}+\frac13\mu_1t^{-1}
  -\frac13\mu_2+ \left(-\frac{1}{3^4}{\mu_1}^3 -\frac1{3^2}\mu_2\mu_1
+\frac13\mu_3 \right)t\\
  &\quad +\left(\frac{1}{3^5}{\mu_1}^4 +\frac1{3^3}\mu_2{\mu_1}^2
  -\frac1{3^2}\mu_3\mu_1 
+\frac1{3^2}{\mu_2}^2-\frac13\mu_4 
\right)t^2\\
  &\quad +\left(-\frac{4}{3^8}{\mu_1}^6-\frac5{3^6}{\mu_2}{\mu_1}^4
  +\frac{5}{3^5}\mu_3{\mu_1}^3+\Big(\frac{1}{3^3}\mu_4
-\frac{2}{3^4}{\mu_2}^2 \Big){\mu_1}^2  \right. \\
  &\quad  \ \ \ \left. +\frac{2}{3^3}\mu_2\mu_3\mu_1-\frac1{3^2}{\mu_3}^2
  -\frac{2}{3^4}{\mu_2}^3+\frac1{3^2}\mu_4\mu_2-\frac13\mu_6\right)t^4
  +O(t^5).
\end{aligned}
\end{equation}
By looking at the recursion relation giving this expansion, we see it
belongs to \(\mathbb{Z}[\tfrac{1}{3},\vec{\mu}][[t]]\).

Throughout this paper, \(\zeta\) is a fixed primitive cube root of
unity.  Transforming \(t \to \zeta t\) and \(t\to\zeta^2t\) gives rise
to similar expansions of \(x^{\star}\) and \(x^{\star\star}\) in terms
of \(t\).  Using the definition of \(\omega\) and a formal reversing
of the function \(t\mapsto v\), we expand the function \(v\mapsto t\).
Substituting this into the expansions of \(t\mapsto x^{\star}\) and
\(t\mapsto x^{\star\star}\) gives expansions 
of
  \(v^{\star}\) and \(v^{\star\star}\) with respect to \(v\) as
\begin{equation}\label{zeta_multiplication}
\begin{aligned}
  v^{\star}     &=\zeta v+\cdots\in\mathbb{Q}[\vec{\mu}, \zeta][[v]], \\
  v^{\star\star}&=\zeta^2 v+\cdots\in\mathbb{Q}[\vec{\mu},\zeta][[v]].
\end{aligned}
\end{equation}

\section{New Addition formula (General form)} 
\label{SEC:AFN}

First, we describe the general structure of our new class of addition formula, 
before constructing explicit examples in the following sections.

We may extend this class of addition formulae by considering more general maps on the curve.  
Let us take a function
\begin{equation*}
  \varphi : \mathscr{C}\longrightarrow \mathbb{P}^1
\end{equation*}
which is a polynomial of \(x\) and \(y\) over \(\mathbb{Z}[\bm{\mu}]\) of
homogeneous weight.  We suppose it is linear in \(y\) and the coefficient of 
its highest weight term with respect to \(x\) and \(y\) (not including \(\{\mu_j\}\)) is \(1\).  
Let \(m \geq 2\) be the order of unique pole of \(\varphi\) and \(u\)
be the analytic variable of \(\varphi\) regarding \(\mathscr{C}\) as a
complex torus.  Then there will exist also conjugate variables
\begin{equation*}
  u, \ u^{\star}, \ u^{\star^2}, u^{\star^3}, \ \cdots, \ u^{\star^{m-1}}.
\end{equation*}
Namely, these \(m\) variables are generically different, vary
continuously, and satisfy
\[
\varphi(u)=\varphi(u^{\star})=\cdots=\varphi(u^{\star^{m-1}}).
\]  
It is clear that these points have similar properties to those in Section \ref{SEC:conjpoints}. 
Namely, that 
\[
u+u^{\star}+\cdots+u^{\star^{m-1}}=0.
\] 
Indeed \(d(u+u^{\star}+\cdots+u^{\star^{m-1}})\) can be regarded as a holomorphic 1-form 
on \(\mathbb{P}^1\) because this varies depending only on \(\varphi(u)\), and hence the vanishing. 
\begin{theorem}
\label{main_gen2}
For \(n\) variables \(u^{(j)}\) \((j=1\), \(\cdots\), \(n)\), under
the conditions stated above
\begin{equation}\label{most_general_lhs}
  \frac{\sigma(u^{(1)}+\cdots+u^{(n)})\prod_{i<j}\prod_{k=1}^{m-1}
\sigma(u^{(i)}+u^{(j)\star^k})}
       {\prod_{j=1}^n\big(
           \sigma(u^{(j)})^{1+(m-1)(n-j)}
           \prod_{k=1}^{m-1}\sigma(u^{(j)\star^k})^{j-1}\big)}
\end{equation}
may be expressed as a polynomial in the \(x(u^{(j)})\) and
\(y(u^{(j)})\) for $j=1$, $\dots$, $n$ of weight
$-\tfrac{1}{2}(n-1)(mn-n+2)$ over the ring
\,\(\mathbb{Q}[\vec{\mu}]\).
\end{theorem}

\begin{remark}
  If $\mu_1=\mu_2=\mu_4=0$, The expression in terms of
  \(x(u^{(j)})\)'s and \(y(u^{(j)})\)'s is symmetric with respect to
  any exchange
\begin{equation*}
  \big(x(u^{(i)}),y(u^{(i)})\big) \longleftrightarrow 
  \big(x(u^{(j)}),y(u^{(j)})\big).  
\end{equation*}
This fact is proved using the following: if \(\mu_1=\mu_2=\mu_4=0\),
we have
\begin{equation*}
  \sigma(\zeta u)=\zeta\sigma(u)
\end{equation*}
(see \cite{emo}, Lemma 4.1), and \(u^{\star}=\zeta u\),
\(u^{\star\star}=\zeta^2 u\).  Then the left hand side is easily shown
to be symmetric with respect to any exchange
\(u^{(i)}\longleftrightarrow u^{(j)}\).
\end{remark}

We prove Theorem \ref{main_gen2} only in the special case \(\varphi=y\) (Theorem \ref{main_gen}).  
The proof of Theorem \ref{main_gen} is sufficiently descriptive to generalise to Theorem \ref{main_gen2}, 
but to write down the full proof for the conjecture would require considerable space and much extra notation, 
without illuminating the general principles involved.  
\begin{theorem}\label{main_gen}
  Let the function \(\varphi=y\), so that \(m=3\).  We denote
  \(u^{{\star}^2}=u^{\star\star}\).  Other
  notation is as introduced above and \(u^{(1)}\), \(u^{(2)}\),
  \(\cdots\), \(u^{(n)}\) are variables. Then
\begin{equation}
\label{general}
  \frac{\sigma(u^{(1)}+u^{(2)}+\cdots+u^{(n)})
    \prod_{i<j}\sigma(u^{(i)}+u^{(j)\star})\,\sigma(u^{(i)} +u^{(j)\star\star})}
  {\prod_{j=1}^n\sigma(u^{(j)})^{2n+1-2j}
   \sigma(u^{(j)\star})^{j-1}
   \sigma(u^{(j)\star\star})^{j-1}}
\end{equation}
may be expressed as a polynomial in the \,\(x(u^{(j)})\) and
\(y(u^{(j)})\) for \(j=1\), ... , \(n\) of weight \(-(n^2-1)\) over
the ring \,\(\mathbb{Q}[\vec{\mu}']\).
\end{theorem}


\begin{remark}\label{int_coeff}
Theorems \ref{main_gen2} and \ref{main_gen} are valid as a power
series identity over quite general base rings and need not be restricted 
only to the case of the complex numbers.
\par
In particular, the coefficients in the expression of
(\ref{most_general_lhs}) in terms of \(x(u^{(j)})\) and
\(y(u^{(j)})\) seem to belong to \(\mathbb{Z}[\vec{\mu}]\).  In the
following sections we show this to be the case for \(\varphi=y\)
with \(n=2\) or \(3\).
\end{remark}


\begin{proof}[Proof of Theorem \ref{main_gen}] 
Regarding (\ref{general}) as a function of each \(u^{(j)}\) we can check that it is meromorphic and 
periodic with respect to \(\Lambda\) (see the proof of Theorem \ref{main_thm} for details of such checks).  
Hence, it must have a rational expression in terms of \(x(u^{(j)})\), \(y(u^{(j)})\) for \(j=1\), ... , \(n\). 
For arbitrarily fixed \(j\), let \(v=u^{(j)}\).  
Then as a function of \(v\), (\ref{general}) has its only pole at \(v=0\) (of order \(2n-1\)).  
Recalling that the $\sigma$-function has weight $+1$ we see that (\ref{general}) has weight $1+n(n-1) - n(2n-1)=-(n^2-1)$. 
So (\ref{general}) can be expressed as a polynomial of the \(x(u^{(j)})\) and \(y(u^{(j)})\) (of weight $-(n^2-1)$) 
and hence an addition formula may be derived by taking (\ref{general}) as the left hand side 
and constructing this polynomial for the right hand side.
 
To find the right hand side we may use the method of
undetermined coefficients as follows.  
Firstly, reducing higher terms of \(y(u^{(j)})\)s in the right hand side to linear terms 
of them by using the relation \(f\big(x(u^{(j)}),y(u^{(j)})\big)=0\), we shall prepare the monomials
\begin{equation}\label{monomial}
  \prod_{j=1}^n{x(u^{(j)})}^{p_j}{y(u^{(j)})}^{\varepsilon_j},
\end{equation}
where \(p_j\) are non-negative with \(2p_j+3\varepsilon_j\leq 2n-1\) and \(\varepsilon_j\) are \(0\) or \(1\).  
Looking at the leading terms in Laurent expansions with respect to \(u^{(j)}\) of these monomials, 
we see that they are linearly independent over \(\mathbb{Q}(\vec{\mu})\).  
Of course, there are only finitely many such monomials.  
Secondly, set the right hand side as
\begin{equation}\label{indet_method}
  \sum_{\{p_j,\varepsilon_j\}}C_{\{p_j,\varepsilon_j\}} \prod_{j=1}^n
  {x(u^{(j)})}^{p_j}{y(u^{(j)})}^{\varepsilon_j}
\end{equation}
with undetermined coefficients \(C_{\{p_j,\varepsilon_j\}}\).  
Because \(\sigma(u^{\star})\) and \(\sigma(u^{\star\star})\) are conjugate each other 
with respect to \(\zeta\longleftrightarrow\zeta^2\), it must be \(C_{\{p_j,\varepsilon_j\}}\in\mathbb{Q}(\vec{\mu})\).  
Then, after rewriting the right hand side by using (\ref{xy_by_sigma}) as a rational function 
of \(\sigma(u^{(j)})\), \(\sigma'(u^{(j)})\), \(\sigma''(u^{(j)})\), \(\sigma'''(u^{(j)})\) for \(j=1\), \(\cdots\), \(n\), 
we multiply both sides by
  \begin{equation*}
  \prod_{j=1}^n\sigma(u^{(j)})^{2n-1}.
  \end{equation*}
Then we get the following equality:
\begin{equation}\label{expand_by_u}
  \begin{aligned}
    \sigma(u^{(1)}&+u^{(2)}+\cdots+u^{(n)})
    \prod_{i<j}\sigma(u^{(i)}+u^{(j)\star})\,\sigma(u^{(i)} +u^{(j)\star\star})\\
    \times & \ \mbox{(a product of power series of \(u^{(j)}\) in 
      \(\mathbb{Q}[\vec{\mu}][[u^{(j)}]]\)\ )}\\
    &=\sum_{\{p_j,\varepsilon_j\}}C_{\{p_j,\varepsilon_j\}} \prod_{j=1}^n
     \sigma(u^{(j)})^{2n-1}\,{x(u^{(j)})}^{p_j}\,{y(u^{(j)})}^{\varepsilon_j}.
  \end{aligned}
\end{equation}
Here, we used that \(\sigma(u^{\star})\sigma(u^{\star\star})/\sigma(u)^2=1+\cdots\in\mathbb{Q}[\vec{\mu}][[u]]\). 
By (\ref{xy_by_sigma}), the right hand side of (\ref{expand_by_u}) is a polynomial of\, \(\sigma(u^{(j)})\),
\(\sigma'(u^{(j)})\), \(\sigma''(u^{(j)})\), \(\sigma'''(u^{(j)})\)\, for \(j=1\),\(\cdots\), \(n\) over \(\mathbb{Q}(\vec{\mu})\).
Using (\ref{hurwitz1}) and that \(u^{\star}u^{\star\star}=u^2+\cdots\in\mathbb{Q}[\vec{\mu}][[u]]\), 
we see that the left hand side of (\ref{expand_by_u}) is 
expanded as a series in
\begin{equation*}
\mathbb{Q}[\vec{\mu}][[u^{(1)},u^{(2)},\cdots,u^{(n)}]]. 
\end{equation*}
Now, we focus on a term of the form
\begin{equation}\label{sigma_mono}
  \prod_{j=1}^n\sigma'(u^{(j)})^{s_j}\,\sigma(u^{(j)})^{k_j}
\end{equation}
for some set \(\{s_j\geq0,\,k_j\geq0\}\) in the right hand side. 
Firstly, we look at such a term with all \(k_j=0\). 
This comes from a unique term of the right hand side of (\ref{expand_by_u}) because, 
by attending the power series expansion of \(\sigma(u)^2\,x(u)\) and \(\sigma(u)^3\,y(u)\) 
with respect to \(u\) obtained by the expressions 
\begin{equation}\label{xy_sigma}
 \begin{aligned}
  x(u)&=\frac{\sigma''(u)\sigma(u)-\mbox{\fbox{\(\sigma'(u)^2\)}}}
             {\sigma(u)^2} \ \ \mbox{and} \\
  y(u)&=\frac{-\tfrac12\sigma'''(u)\sigma(u)^2+\tfrac32\sigma''(u)
    \sigma'(u)\sigma(u)-\mbox{\fbox{\(\sigma'(u)^3\)}}}{\sigma(u)^3}\\
  &\qquad\qquad-\frac{\mu_1}2{\cdot}\frac{\sigma''(u)\sigma(u)-\sigma'(u)^2}
  {\sigma(u)^2}+\frac{\mu_3}2, 
 \end{aligned} 
\end{equation}
the lowest terms of the power series expansion of a term in 
the right hand side of (\ref{expand_by_u}) with respect to \(u^{(j)}\)s 
will be contributed by the boxed terms. 
Hence the coefficient \(C_{\{p_j,\varepsilon_j\}}\) of the unique term is in \(\mathbb{Z}[\vec{\mu}']\). 
In the next time, we introduce an order to the set 
\[
\{k_1,k_2,\cdots,k_n;s_1,s_2,\cdots,s_n\}
\]
with lexicographic order in \(k_j\)s and anti-lexicographic order in \(s_j\)s, and with 
assuming the former order is stronger than latter. 
According to this order, we check successively that each term (\ref{sigma_mono}) comes from 
which terms in the right hand side of (\ref{expand_by_u}) and we see the corresponding coefficients 
\(C_{\{p_j,\varepsilon_j\}}\) are all in \(\mathbb{Q}[\vec{\mu}]\). 
\end{proof}


\section{New Addition formula (2-variable case)}
\label{SEC:AF2}

\noindent 
For a fixed \(x\), we have two points on the curve.  If one point is denoted say \((x,y)\), then the other point is \((x,-y-\mu_1x-\mu_3)\).  
In this situation, if 
\(u=\int_{\infty}^{(x,y)}\omega\) 
then 
\(-u=\int_{\infty}^{(x,-y-\mu_1x-\mu_3)}\omega\).  
Suppose we replace the sigma function in equation (\ref{2term_fs}) from the Introduction 
by the most general sigma function from Definition \ref{def:gensig}.  
It can be easily checked that (\ref{2term_fs}) remains valid for the fully general curve \(\mathscr{C}\).

We now give the our first explicit new addition formula, 
by considering fixing the other coordinate and using the conjugate variables defined in Section \ref{SEC:conjpoints}.
We use the notation of the previous sections but with variables $u$ and $v$ in place of the $u^{(1)}$ and $u^{(2)}$ from Section \ref{SEC:AFN}.

\begin{theorem}\label{main_thm} 
We have the addition formula
\begin{equation}\label{main}
  \begin{aligned}
    - \frac{\sigma(u+v)\sigma(u+v^{\star})\sigma(u+v^{\star\star})}
    {\sigma(u)^3\sigma(v)\sigma(v^{\star})\sigma(v^{\star\star})}
    &=y(u)-y(-v)\\
    &=y(u)+y(v)+\mu_1x(v)+\mu_3\\    
    &\hskip -30pt
     =\frac12\big(\wp'(u)+\wp'(v)\big)
     +\frac{\mu_1}{2}\big(\wp(u)-\wp(v)\big).\\
  \end{aligned}
\end{equation}
\end{theorem}

\begin{remark}\label{on_classical}  
We first comment on how our formula is modified when specialising the curve. 
\begin{enumerate}
\item As noted earlier, when the parameters in (\ref{ellipt_curve}) take values 
\(\mu_1=\mu_2=\mu_3=0\),  \(\mu_4=-\tfrac{1}{4}g_2\), \(\mu_6=-\tfrac{1}{4}g_3\),
the function \(\wp(u)\) from Definition \ref{def:wp} satisfies the classical equation  (\ref{weier_ellipt}).  
In this case, by (\ref{parity}) and (\ref{P_and_xy}), the right hand side of the formula reduces to give 
the addition formula (\ref{emo_fs}) we presented in the introduction.
\item If we specialise further to consider the equianharmonic case (by further setting further $\mu_4=g_2=0$) then 
Theorem \ref{main_thm} reduces to Proposition 5.1 of \cite{emo}, with equation (\ref{main}) becoming (\ref{emo_fs}) from the Introduction.
\end{enumerate}
\end{remark}

\begin{proof}[Proof of Theorem \ref{main_thm}] 
We first prove the left hand side of (\ref{main}) is a meromorphic function of both  \(u\) and \(v\).  
Using (\ref{sum_of_v}) and Lemma \ref{translation}, 
we see the left hand side is invariant with respect to the transformations \(u\mapsto u+\ell\), 
\(v\mapsto v+\ell\) for \(\ell\in\Lambda\).  For the transformation \(u\mapsto u+\ell\), the exponent of the exponential factor becomes
  \begin{equation*}
  \begin{aligned}
    &L(u+v+\tfrac12\ell,\ell)+L(u+v^{\star} +\tfrac12\ell,\ell) \\
    &\qquad +L(u+v^{\star\star}+\tfrac12\ell,\ell)
    -3\,L(u+\tfrac12\ell,\ell)\\
    &=L(v+v^{\star}+v^{\star\star}, \ell)\\
    &=L(0,\ell)\\
    &=0.
  \end{aligned}
  \end{equation*}
For \(v\mapsto v+\ell\), it becomes 
\begin{equation*}
  \begin{aligned}
    &\ \ \ L(u+v+\tfrac12\ell,\ell)+L(u+v^{\star}
    +\tfrac12\ell^{\star},
    \ell^{\star})+L(u+v^{\star\star}+\tfrac12\ell^{\star\star},
    \ell^{\star\star})\\
    &\ \ \ \ \ \
    -L(v+\tfrac12\ell,\ell)-L(v^{\star}+\tfrac12\ell^{\star},
    \ell^{\star}) -L(v^{\star\star}+\tfrac12\ell^{\star\star},
    \ell^{\star\star})\\
    &=L(u, \ell)+L(u, \ell^{\star})+L(u, \ell^{\star\star})\\
    &=L(u, \ell+\ell^{\star}+\ell^{\star\star})\\
    &=L(u,0)\\
    &=0.
  \end{aligned}
\end{equation*}
Therefore, the left hand side is a function of \(u\) modulo \(\Lambda\).  It also has a unique pole at \(u=0\).  
It is well-known that such a function is a polynomial of \(\wp(u)\) and its higher order derivatives.  
In this case the poles are of order \(3\), so we need only use $\wp$ and $\wp'$. 
 
Since the equation must be of homogeneous weight (weight \(-3\) on both sides), we know that the left hand side must be of the form
\begin{equation*}
  a_1\wp'(u)+a_2\wp'(v)+b_1\mu_1\wp(u)+b_2\mu_1\wp(v)+c_1{\mu_1}^3 
  +c_2{\mu_1}\mu_2+c_3\mu_3
\end{equation*}
for some constants \(a_1\), \(a_2\), \(b_1\), \(b_2\), \(c_1\), \(c_2\) and \(c_3\).  
For arbitrary fixed \(v\), as a function of \(u\), the left hand side has zeroes 
at \(u=-v\), \(u=-v^{\star}\), \(u=-v^{\star\star}\) (of order \(1\) each), and no other zeros.  
Using the fact that the $\wp(u)$ is an even function we have that
\begin{equation*}
  a_2=a_1\ (= a \mbox{ say}), \quad 
  -b_2=b_1\ (=b \mbox{ say}), \quad 
  c_1=c_2=c_3=0.
\end{equation*}
Substituting the truncated expansion (\ref{x_exp}) up to the constant term and (\ref{zeta_multiplication}) into (\ref{main}) gives
\begin{equation*}
  -\frac{1}{u^3}-\frac{1}{v^3}+\frac12\mu_1\bigg(\frac{1}{u^2} 
  -\frac{1}{v^2}\bigg)+\cdots.
\end{equation*}
Since 
\begin{equation*}
\wp(u)=\frac{1}{u^2}+\cdots,
\end{equation*}
we find the coefficients are $a=\frac{1}{2}$ and $b=\frac{1}{2}$, concluding the proof.
%
\end{proof}
We finish the section with some further remarks on the new formula (\ref{emo_fs}).  
It could be argued that this formula lacks symmetry as the variables $u$ and $v$ are treated differently.  
We can replace $u$ by $u^{\star}$ and $u^{\star\star}$ in turn, remembering that
$\wp(u)=\wp(u^{\star})=\wp(u^{\star\star})$, then add the three to get
\[  
\sum_{i=1}^3\left[\frac{\prod_{j=1}^3\sigma(u_i+v_j)}
  {\sigma(u_i)^3\prod_{j=1}^3 \sigma(v_j)}\right] = \frac32
    \left(\wp'(u)+\wp'(v)\right)
\]
where for typographical convenience we use \(u_i,
  i=1,2,3\), to represent \(u, u^{\star}\), and \(u^{\star\star}\) respectively. 
However in producing such a formula we are throwing away information,
in particular by subtracting two of the three relations described
above we can get
\[
\frac{\sigma(u+v)\sigma(u+v^{\star})\sigma(u+v^{\star\star})}{\sigma(u)^3}
=
\frac{\sigma(u^{\star}+v)\sigma(u^{\star}+v^{\star})\sigma(u^{\star}+v^{\star\star})}
{ \sigma(u^{\star})^3},
\]
and similarly for $(u,u^{\star})$ and
$(u^{\star},u^{\star\star})$. 
A similar equation is seen in Corollary 12.2 of \cite{o11}.

\section{New Addition formula (3-variable case)}
\label{SEC:AF3} 

The second new explicit addition formula is given below.  
It is a natural three variable extension of Theorem \ref{main_thm}.  
See also \cite{fs} and \cite{o11} for similar formulae.

\begin{theorem}\label{main_thm3}
Let \(u\), \(v\), and \(w\) be variables.  Denote, 
for brevity, \((x_u,y_u)=(x(u),y(u))\) and similarly for $v$ and $w$.
With the notation of the previous sections we have a new addition formula expressing
\begin{equation*}
\frac{\sigma(u+v+w)
\sigma(u+v^{\star})\sigma(u+v^{\star\star})
\sigma(u+w^{\star})\sigma(u+w^{\star\star})
\sigma(v+w^{\star})\sigma(v+w^{\star\star})}
{\sigma(u)^5
 \sigma(v)^3\sigma(v^{\star})\sigma(v^{\star\star})
 \sigma(w)\sigma(w^{\star})^2\sigma(w^{\star\star})^2}
\end{equation*}
as $\sum_{i=0}^8 r_i$ where the $r_i$ are as stated below. Each is a polynomial in \(x_u\), \(x_v\),
\(x_w\), \(y_u\), \(y_v\), \(y_w\), and the \(\{\mu_j\}\) {\rm(}of combined
weight \(i\){\rm)}.
\begin{equation*}
\begin{aligned}
 r_0 &=(y_u y_v + y_u y_w + y_v y_w - x_u x_v x_w)(x_u +x_v +x_w) \\
 &\qquad 
 -x_u^2 x_v^2 -x_u^2 x_w^2 -x_v^2 x_w^2, \\
  r_1 &= \mu_1 (x_v x_u y_v +2 x_v x_u y_w +2 y_w x_u^2 +x_w x_u y_w
  -x_w^2 y_u \\ &\qquad +x_v x_u y_u 
  +x_w y_v x_u +y_v x_u^2 +y_w x_v^2),
  \\
  r_2 & = ( x_u^2 x_v -x_u x_w^2 +y_w y_u) \mu_1^2 -(x_v^2 x_w -y_v
  y_u +x_u^2 x_v +x_u x_w^2 \\ 
  &\qquad +2 x_v x_w x_u 
  -y_w y_u -y_w y_v +x_v x_w^2 +x_u^2 x_w +x_u x_v^2) \mu_2,
  \\
  r_3 &= \mu_1^3 y_w x_u +(x_u y_v +2 y_w x_u +x_v y_w -x_w y_u) \mu_2\mu_1
  \\  &\quad 
  +(y_v +y_w +y_u) (x_u +x_v +x_w) \mu_3, 
  \\
  r_4 &= -\mu_1^2 x_u \mu_2 x_w +( x_u^2 -x_w^2 +2 x_u x_v +x_u x_w)
  \mu_3 \mu_1 
  \\ &\qquad 
  -(x_u x_v +x_v x_w +x_u x_w) \mu_2^2
  -(x_u^2 +x_v^2 +x_w^2) \mu_4,
  \\
  r_5 &= \mu_1^2 y_w \mu_3 -( y_u-y_w ) \mu_4 \mu_1
  +(y_v +y_w +y_u) \mu_3 \mu_2,
  \\
  r_6 &= (x_u -x_w ) \mu_3 \mu_2 \mu_1 -\mu_1^2 x_u \mu_4 
  -(x_u +x_v +x_w) (\mu_2 \mu_4 -\mu_6 -\mu_3^2),
  \\
  r_7 &= 0, \\
  r_8 &= (\mu_6 +\mu_3^2) \mu_2-\mu_1 \mu_3 \mu_4  -\mu_4^2.
\end{aligned}
\end{equation*}
\end{theorem}

\begin{proof}
The left hand side of the new formula is meromorphic in \(u\), \(v\), and \(w\).  
Moreover, we can check easily that it is periodic with respect to \(\Lambda\).  
Hence it may be expressed in terms of elliptic functions.  
Further, we can check that the left hand side has poles of order five each in $u,v$ and $w$ and so the right hand side must have an expression in $\wp(u), \wp(v)$ and $\wp(w)$ and their derivatives up to third order.  
More specifically, the right hand side will be a sum of terms, each a product of three functions, one in each of the variables and with all functions taken from the set $\{1,\wp,\wp',\wp'',\wp'''\}$.  Such an expression is clear from the linear algebra when considering the space of elliptic functions graded by pole order (for more details on such spaces see for example \cite{eeo11, ea12}).  
This also clarifies why $r_7=0$: since there is no elliptic function of weight 1 to include in the right hand side.
 
The coefficients of this right hand side may then be determined using the series expansions of the functions discussed earlier.  Since the left hand side is of weight $-8$ the expansions used need to contain terms with monomials in $\mu_i$ up to weight $-8$.  We used {\sc Maple} to implement this calculation (with details on similar calculations given in \cite{eeo11}).  The right hand side presented above was then obtained by making the substitutions implied by (\ref{P_and_xy}).
\end{proof}
\begin{remark}
Using the mappings in {\rm (\ref{P_and_xy})} we could rewrite the right hand side of the formula in Theorem \ref{main_thm3} in terms of \(\wp\) and its first derivative. 
\end{remark}

\begin{remark}
Let
\begin{equation*}
\begin{aligned}
f_2&=x_{u}+x_{v}+x_{w}+\mu_2,\\
f_4&=x_{u}x_{v} + x_{v}x_{w} + x_{u}x_{w}-\mu_4+\mu_1 y_{w},
\end{aligned}
\end{equation*}
where the suffices of $f$ are chosen to denote the weight.  Each of these vanishes when \(v=u^{\star}\) and \(w=u^{\star\star}\) at the same time, since then \(y(u)=y(u^{\star})=y(u^{\star\star})\) and \(x(u)\), \(x(u^{\star})\), \(x(u^{\star\star})\) are the three solutions of the cubic equation
\begin{equation*}
X^3+\mu_2X^2+(\mu_4-\mu_1y(u))X+\mu_6-y(u)^2-\mu_3y(u)=0.
\end{equation*}
A calculation with Gr\"obner bases implemented with {\sc Maple} shows that the right hand side of the formula presented in Theorem \ref{main_thm3} lies in the ideal generated by $f_2$ and $f_4$.   Specifically, we have
\begin{equation*}
\sum_{i=0}^8 r_i = Q_6 f_2+ Q_4 f_4, 
\end{equation*}
where
\begin{align*}
Q_6 &= y_w \mu_1^3 -(\mu_4 -x_v x_w -x_u x_v) \mu_1^2 
\\ &\quad 
+(x_u \mu_3 -\mu_3 x_w -x_w y_w  -x_w y_u +2 y_w x_u  + x_u y_v )\mu_1 \\ 
&\quad -(x_u x_v +x_v x_w +x_u x_w) \mu_2 +\mu_3^2
 +(y_v +y_w +y_u) \mu_3 -x_u x_v^2 +\mu_6 \\ 
 &\quad - x_u x_w^2 +y_w y_u
 -x_v^2 x_w +y_v y_u -x_v x_w^2 +y_w y_v -x_v x_w x_u -x_u \mu_4,\\
Q_4 &=  (y_u +\mu_3) \mu_1-(\mu_2 +x_v +x_w) \mu_1^2 +(x_v +x_w) \mu_2 
\\ &\quad 
+\mu_4 +x_w^2 +x_v x_w +x_v^2.
\end{align*}
This expression, along with (\ref{star_relation}) shows that both sides
of the equation in Theorem \ref{main_thm3} vanish when \(v=u^{\star}\) and
\(w=u^{\star\star}\).  
\end{remark}

\begin{remark}
In Remark \ref{on_classical} we discussed how the \(2\)-variable formula collapsed to known results when restricting the curve.  
We note now some similar restrictions for the \(3\)-variable result.
\begin{enumerate}
\item  If \(\mu_1=\mu_2=\mu_3=0\), \(\mu_4=-\tfrac14g_2\), \(\mu_6=-\tfrac14g_3\) 
in (\ref{ellipt_curve}), then the right hand side of the formula in Theorem \ref{main_thm3} becomes 
\begin{align}
  -\tfrac{1}{16} g_2^2&+\tfrac14 g_2 ({\wp(v)}^2 +{\wp(w)}^2 +{\wp(u)}^2)
  \\ &
  -{\wp(u)}^2{\wp(w)}^2 -{\wp(v)}^2{\wp(w)}^2 -{\wp(u)}^2{\wp(v)}^2  \nonumber \\
  &-\tfrac{1}{4}(\wp(u) + \wp(v) + \wp(w)) \big(4\wp(u) \wp(v) \wp(w) + g_3
  \nonumber \\  &\qquad  
  -\wp'(u) \wp'(v) -\wp'(v) \wp'(w) -\wp'(u) \wp'(w)\big).   \label{eq:3TRem1}
\end{align}
\item Suppose instead we simplify by setting $\mu_1=\mu_2=\mu_4=0$.  Of course we get another simplification of the right hand side, but in this case also a simplification of the left hand side.  
Now, $x^3$ is the only term in the curve equation with $x$ and so the starred   variables can all be described using roots of unity acting on the non-starred variables.  Hence in this case we have 
\begin{equation}\label{eq:3TRem2}
\begin{aligned}
&\frac{\sigma(u+v+w)
\sigma(u{+}\zeta v)\sigma(u{+}\zeta^2 v)
\sigma(u{+}\zeta w)\sigma(u{+}\zeta^2 w)
\sigma(v{+}\zeta w)\sigma(v{+}\zeta^2 w)}
{\sigma(u)^5
 \sigma(v)^3\sigma(\zeta v)\sigma(\zeta^2 v)
 \sigma(w)  \sigma(\zeta w)^2\sigma(\zeta^2 w)^2} \\
&= (x_{v}+x_{u}+x_{w})\mu_{6} + (x_{v}+x_{u}+x_{w})\mu_{3}^2\\
&\quad + (y_{u}+y_{v}+y_{w})(x_{v}+x_{u}+x_{w})\mu_{3} - x_{u}^2x_{w}^2 - x_{v}^2x_{w}^2 - x_{u}^2x_{v}^2 \\
&\quad - (x_{v}+x_{u}+x_{w})(x_{v}x_{w}x_{u} - y_{v}y_{w} - y_{u}y_{w} - y_{v}y_{u}).
\end{aligned}
\end{equation}

\item The equianharmonic case is a sub-case of the both the above specialisations.  
In this case we have the simplified left hand side from (\ref{eq:3TRem2}) 
and a further reduced right hand side which obtained by setting $g_2=0$ in (\ref{eq:3TRem1}).  
Using $\wp$-coordinates analogously to (\ref{emo_fs}), the right hand side is
\begin{align*}
&\tfrac{1}{4}(\wp(u) +\wp(v) +\wp(w))( \wp'(u) \wp'(v)
+ \wp'(v) \wp'(w) +\wp'(u) \wp'(w)  \\
&\quad -g_3 - 4\wp(v)\wp(w)\wp(u)) -\wp(u)^2 \wp(v)^2 -\wp(u)^2 \wp(w)^2 -\wp(v)^2 \wp(w)^2.
\end{align*}
\end{enumerate}
\end{remark}

\section{Final remarks} 
\label{SEC:Final}

A specialisation not considered above was the rational case, i.e. setting all $\mu_i=0$ (and $g_i=0$).  In this case it may be checked that all equations collapse to simple algebraic identities.

We finish by giving some thoughts on further generalisations of the results.  
\begin{enumerate}

\item As proved by Theorem \ref{main_gen}, the explicit formulae certainly generalise to an $n$-variable case.
However, we find that trying to derive the expanded form of the right hand side in the 4-variable case 
using naive series expansions exceeds the memory limits of the current machines available to us.  
We expect that progress would follow from the discovery of a more compact expression for these right hand sides, for example, as a determinant.

\item For the equianharmonic curve \(y^2=x^3+\mu_6\), there is an action of the group of the sixth roots of unity acts on this curve, and on the coordinate space \(\mathbb{C}\) of \(\wp(u)\) and \(\sigma(u)\).  Let \(\zeta=\exp(2\pi \rm{i} /3)\), a third root of unity. In \cite{emo}, we gave a 3-variable formula giving
  \begin{equation*}
    \frac{\sigma(u+v+w)\sigma(u+{\zeta}v+{\zeta^2}w) 
      \sigma(u+{\zeta}^2v+{\zeta}w)}
    {\sigma(u)^3\sigma(v)^3\sigma(w)^3}
  \end{equation*}
as a polynomial of \(\wp(u)\), \(\wp(v)\), \(\wp(w)\), and their first order derivatives.  
  
Thus it is reasonable to consider a naive generalisation of this in our setting, namely and expression for
  \begin{equation*}
    \frac{\sigma(u+v+w)\sigma(u+v^{\star}+w^{\star\star})
      \sigma(u+v^{\star\star}+w^{\star})}
    {\sigma(u)^3\sigma(v)\sigma(v^{\star})\sigma(v^{\star\star})
      \sigma(w)\sigma(w^{\star})\sigma(w^{\star\star})}.
  \end{equation*}
However, we find this is no longer a periodic function with respect to  \(\Lambda\), 
as may be checked by the translational formula (\ref{translational}).  
If we increase \(v\) to \(v+\ell\) (and similarly for $w$), the factors which appear in (\ref{translational}) do not cancel out.
\item Our results are likely to generalise to higher genus curves.  
For example, the natural analogue for Theorem \ref{main_thm} for the curve
  \begin{equation*}
    y^2+(\mu_1x^2+\mu_3x+\mu_5)y
    = x^5+\mu_2x^4+\mu_4x^3+\mu_6x^2+\mu_8x+\mu_{10}
  \end{equation*}
could be obtained by considering five roots of \(x\) for a fixed \(y\).  
\end{enumerate}

\section*{Acknowledgements}
This work began following the presentation
``Frobenius-Stickelberger type formulae for general curves'' by Y\^O
at the 2010 ICMS conference entitled ``The higher-genus sigma
function and applications''.  It also follows the work in \cite{o11}
and \cite{emo}.  

The material in this paper was mainly derived when JCE and ME visited
Y\^O at the University of Yamanashi in Spring 2012, supported by JSPS
grant no.22540006.  

The authors acknowledge the anonymous referees whose comments improved the paper.  
In particular we thank the referee who pointed out the extension of Theorem \ref{main_gen} to Theorem \ref{main_gen2}.



\end{document}